\documentclass[hidelinks]{amsart}

\pdfoutput=1

\usepackage{a4wide}
\usepackage{amsmath}
\usepackage{stackrel}
\usepackage{amssymb}
\usepackage{graphicx}
\usepackage[numbers]{natbib}
\usepackage[english]{babel}
\usepackage[utf8]{inputenc}
\usepackage[T1]{fontenc}
\usepackage{hyperref}
\usepackage{enumitem}
\usepackage{multicol}
\usepackage{mathtools}
\usepackage{multirow}
\usepackage{array}
\usepackage{booktabs}
\usepackage{tikz}
\usetikzlibrary{arrows}
\usetikzlibrary{shapes}
\usetikzlibrary{positioning}
\usetikzlibrary{cd}
\usetikzlibrary{babel}
\usepackage[toc,page]{appendix}
\usepackage[all]{xy}
\usepackage{float}
\usepackage{comment}
\usepackage{centernot}
\usepackage{pifont}
\usepackage[colorinlistoftodos]{todonotes}

\newcommand{\drl}{\partial_{\mathrm{rl}}}
\newcommand{\dSets}{\mathsf{dSets}}

\newcommand{\Fdec}{\mathbb{F}^{\text{dec}}}
\newcommand{\Fin}{\mathsf{Fin}}

\newcommand{\sSets}{\mathsf{sSets}}

\newcommand{\Prop}{\mathsf{Env}}
\newcommand{\Proptensor}{\mathsf{Env}}

\newcommand{\Tree}{\mathsf{\Omega}}

\newtheorem{theorem}{Theorem}[section]
\newtheorem{lemma}[theorem]{Lemma}
\newtheorem{proposition}[theorem]{Proposition}

\theoremstyle{definition}

\newtheorem{definition}[theorem]{Definition}
\newtheorem{example}[theorem]{Example}
\newtheorem{remark}[theorem]{Remark}
\newtheorem*{theorem*}{Theorem}
\newtheorem*{lemma*}{Lemma}
\newtheorem*{notation*}{Notation}
\newtheorem*{proposition*}{Proposition}
\newtheorem*{corollary*}{Corollary}
\newtheorem*{definition*}{Definition}
\newtheorem*{remark*}{Remark}

\setlist[itemize]{leftmargin=0.8cm}
\setlist[enumerate]{leftmargin=0.8cm}
\makeatletter
\renewcommand{\@biblabel}[1]{[#1]\hfill}
\makeatother

\makeatletter
\providecommand\@dotsep{5}
\renewcommand{\listoftodos}[1][\@todonotes@todolistname]{%
  \@starttoc{tdo}{#1}}
\makeatother

\setlist[itemize]{label=$\diamond$, leftmargin=0cm, itemindent=0.8cm}
\setlist[enumerate]{leftmargin=0cm, itemindent=1cm}
\makeatletter
\renewcommand{\@biblabel}[1]{[#1]\hfill}
\makeatother

\makeatletter
\def\thmhead@plain#1#2#3{%
  \thmname{#1}\thmnumber{\@ifnotempty{#1}{ }\@upn{#2}}%
  \thmnote{ {\the\thm@notefont#3}}}
\let\thmhead\thmhead@plain
\makeatother

\makeatletter
\providecommand\@dotsep{5}
\renewcommand{\listoftodos}[1][\@todonotes@todolistname]{%
  \@starttoc{tdo}{#1}}
\makeatother

\makeatletter
\renewcommand{\tocsection}[3]{%
  \indentlabel{\@ifnotempty{#2}{\bfseries\ignorespaces#1 #2\quad}}\bfseries#3}
\renewcommand{\tocsubsection}[3]{%
  \indentlabel{\@ifnotempty{#2}{\ignorespaces#1 #2\quad}}#3}

\def\@tocline#1#2#3#4#5#6#7{\relax
  \ifnum #1>\c@tocdepth 
  \else
    \par \addpenalty\@secpenalty\addvspace{#2}%
    \begingroup \hyphenpenalty\@M
    \@ifempty{#4}{%
      \@tempdima\csname r@tocindent\number#1\endcsname\relax
    }{%
      \@tempdima#4\relax
    }%
    \parindent\z@ \leftskip#3\relax \advance\leftskip\@tempdima\relax
    \rightskip\@pnumwidth plus1em \parfillskip-\@pnumwidth
    #5\leavevmode\hskip-\@tempdima{#6}\nobreak
    \leaders\hbox{$\m@th\mkern \@dotsep mu\hbox{.}\mkern \@dotsep mu$}\hfill
    \nobreak
    \hbox to\@pnumwidth{\@tocpagenum{\ifnum#1=1\bfseries\fi#7}}\par
    \nobreak
    \endgroup
  \fi}
\AtBeginDocument{%
\expandafter\renewcommand\csname r@tocindent0\endcsname{0pt}
}
\def\l@subsection{\@tocline{2}{0pt}{2.5pc}{5pc}{}}
\makeatother

\usepackage{srcltx}

\begin{document}

\title{The right cancellation property for certain classes of dendroidal anodynes}

\author{Miguel Barata}
\address{Universiteit Utrecht, Hans Freudenthalgebouw, Budapestlaan 6, 3584 CD Utrecht, The Netherlands}
\email{m.lourencohenriquesbarata@uu.nl}

\maketitle

\vspace{-2.5em}

\begin{abstract}
\leftskip7em
\rightskip7em
   We generalize a previous result of Stevenson \cite{StevensonInnerFibrations} to the category of dendroidal sets, yielding the right cancellation property of dendroidal inner anodynes within the class of normal monomorphisms. As an application of this property, we show how to construct a symmetric monoidal $\infty$-category $\Prop(X)^\otimes$ from a dendroidal $\infty$-operad $X$, in a way that generalizes the symmetric monoidal envelope of a coloured operad. 
\end{abstract}

\vspace{1.5em}

\tableofcontents

\section{Introduction}

The purpose of this note is to give a generalization of the right cancellation property of inner anodynes in the category of simplicial set $\sSets$ to the category of dendroidal sets $\dSets$. We register here the central definition of this work.

\begin{definition}
    Let $\mathcal{M}$ and $\mathcal{N}$ be classes of morphisms of a category $\mathcal{C}$ such that $\mathcal{M}$ is contained in $\mathcal{N}$. Then we say that $\mathcal{M}$ satisfies the \textit{right cancellation property} within $\mathcal{N}$ if given any string of morphisms in $\mathcal{C}$
    $$ A \xrightarrow{i} B \xrightarrow{j} C$$
    such that $i, j, ji$ are in $\mathcal{N}$ and additionally both $i$ and $ji$ are in $\mathcal{M}$, then $j$ is also in $\mathcal{M}$.    
\end{definition}

A knowledge of which classes of simplicial morphisms satisfy the right cancellation property seems to be of particular interest, for instance to simplify certain conbinatorial arguments in $\sSets$. As an example of this phenomenon, one has the well-known result \citep[Prop. 5.34]{HeutsMoerdijkDendroidal} in simplicial homotopy theory which guarantees that, for any saturated class of simplicial morphisms $\mathcal{M}$ satisfying the right cancellation property within the monomorphisms, if all the spine inclusion are in $\mathcal{M}$, then actually all inner anodynes are in $\mathcal{M}$. The upshot of such a result is that, under the circumstances described, one can often argue with spine inclusions in place of inner horn inclusions, which tends to be combinatorially easier to do.

For the sake of easy reference, we register below the simplicial statement we are aiming to generalize.

\begin{theorem}
    The class of inner anodynes in $\sSets$ satisfies the right cancellation property within the class of simplicial monomorphisms.
    \label{right_cancel_simplicial}
\end{theorem}

The first account and proof of Theorem \ref{right_cancel_simplicial} appears in \cite{StevensonInnerFibrations} as Theorem 1.5 and seems to have gone mostly unnoticed in the literature on simplicial sets. In contrast, the right cancellation property for simplicial left and right anodynes is more well-known: it was already present in Joyal's work on simplicial sets and quasicategories, for instance. 

In what concerns the proof of the right cancellation property for simplicial left and right anodyne maps, the interested reader can check \citep[Cor. 4.1.2.2]{LurieHTT} and \citep[Cor. 8.80]{HeutsMoerdijkDendroidal} for instance. In any case, we want to mention that these proofs are not entirely trivial, and rely on certain properties of left and right anodynes deduced from the existence of the contravariant and covariant model structures on the slice categories $\sSets_{/ V}$. In contrast, this argument cannot be followed through in the inner anodyne scenario, due to the corresponding analogues of the contravariant and covariant model structures not existing. In this sense, Stevenson's proof turns to be a better approach: it actually also works (and simplifies somewhat) if one wants to show the right cancellation property for left and right anodynes.

As we have already mentioned, our intention is to generalize Stevenson's result to the context of dendroidal sets. This is a category of presheaves on the tree category $\Tree$, first introduced by Moerdijk and Weiss in \cite{MoerdijkWeissDendroidal}, and this category play a role in the theory of operads analogous to the one played by $\sSets$ in category theory. After giving a brief account in Section 2 of some notation and results on dendroidal sets, we will dedicate Section 3 to proving the following generalization of Theorem \ref{right_cancel_simplicial}.

\begin{theorem}
    Inner/leaf/root anodynes in $\dSets$ satisfy the right cancellation property within the class of normal monomorphisms.
    \label{right_cancel_dendroidal}
\end{theorem}

Here the classes of leaf and root anodynes play a similar role in the dendroidal formalism as that of left and right anodynes in $\sSets$, respectively. We remark that the right cancellation property for leaf anodynes is already known to hold: for instance, one can easily adapt the simplicial argument presented in \cite{HeutsMoerdijkDendroidal} to the dendroidal context, using the existence of the version of the covariant model structure for the overcategories $\dSets_{/V}$. However, the version of Theorem \ref{right_cancel_dendroidal} about inner and root anodynes doesn't seem to be present in the literature\footnote{We note that, for the simplicial context, the result for left anodynes immediately implies the result for right anodynes, due to the existence of the involution $\mathsf{\Delta} \to \mathsf{\Delta}$ which reverses the order of each poset. However, there is no parallel in $\Tree$ of this involution, and therefore the leaf anodyne result doesn't imply the root anodyne one.}.

In Section 4 we give an application of the right cancellation property of dendroidal inner anodynes. We explain how one can assign to a dendroidal $\infty$-operad $X$ a simplicial set $\Proptensor(X)^\otimes$ which recovers the usual construction of the symmetric monoidal envelope of a coloured operad. The exact result we will show is the following:

\begin{theorem}
    Let $X$ be a dendroidal set. Then there exists a simplicial set $\Proptensor(X)^\otimes$ satisfying the following properties:
    \begin{enumerate}[label=(\alph*)]
        \item If $X$ is an $\infty$-operad, then $\Proptensor(X)^{\otimes}$ is a symmetric monoidal $\infty$-category in the sense of Lurie.
        \item If $\mathbf{P}$ is a coloured operad and $N\mathbf{P} \in \dSets$ is the corresponding dendroidal nerve, then $\Proptensor(N\mathbf{P})^\otimes$ is the symmetric monoidal $\infty$-category associated to the symmetric monoidal category $\Proptensor(\mathbf{P})$ corresponding to the symmetric monoidal envelope of $\mathbf{P}$.
    \end{enumerate}
    \label{Prop_main_theorem}
\end{theorem}

The main technical hurdle for proving this result will be handled in Section 4.1, where we use the right cancellation property of dendroidal inner anodynes to investigate a certain category of forests $\mathbb{F}$ constructed from $\Fin_\ast$. Finally, in Section 4.2 we will use these dendroidal considerations to construct $\Proptensor(X)^\otimes$ and prove Theorem \ref{Prop_main_theorem}.

\vspace{0.2cm}

\noindent \textbf{Acknowledgements.} The author was supported by the Dutch Research Council (NWO) through the grant 613.009.147. The author would like to thank Ieke Moerdijk for helpful feedback on earlier drafts of this work.

\section{Some recollections about dendroidal sets}

Before embarking on proving Theorem \ref{right_cancel_dendroidal}, we will first quickly recall some definitions and notation regarding the tree category $\Tree$, as well as some technical results that will be used in the rest of the text. For a textbook account of the subject and a more thorough presentation of some of the definitions below, see \cite{HeutsMoerdijkDendroidal}.

The tree category $\Tree$ has objects given by non-planar rooted trees $T$ with finite vertex and edge sets, together with a specified special edge which is attached to a single vertex, and we call this edge the \textit{root} of $T$. As explained in Section 1.3 of \cite{HeutsMoerdijkDendroidal}, every tree $T$ gives rise to a coloured operad $\Tree(T)$, and we define the morphism set $\Tree(T,S)$ as the the set of all operad maps $\Tree(T) \to \Tree(S)$. We specify some notation which is useful when discussing trees:
\begin{itemize}[label=$\diamond$]
    \item Any edge in $T$ that is adjacent to two vertices is called an \textit{inner edge}; the edges which are neither inner nor the root edge are called the \textit{leaves}.
    \item The edge set $E(T)$ has a poset structure by setting $e_1 \leq e_2$ when the (unique) path from $e_2$ to the root edge contains $e_1$. The minimal element of this ordering is the root edge, and we call the maximal edges the \textit{leaves} of $T$. A similar ordering can also be defined on the vertex set $V(T)$, with the minimal element being the \textit{root vertex} and the maximal elements being the \textit{leaf vertices}.
    \item The set of edges incident to a vertex $v$ has a unique minimal element called the \textit{outcoming edge} of $v$, and the remaining edges, which are pairwise independent in the ordering above, are the \textit{incoming edges} of $v$. In the scenario when $v$ admits no incoming edges, we say that the vertex $v$ is a \textit{stump}.
    \item the tree which is just a single edge will be represented by $\eta$. The trees with a single vertex and $n$ incoming edges are the \textit{$n$-corollas} and will be denoted throughout by $C_n$. As an example, here is a picture of $C_4$:
    
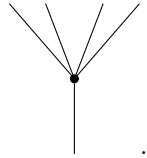
\vspace{0.2cm}
    \begin{center}
\begin{tikzpicture}
    \draw ({-sin(22.5)},1)--(0,0)--({sin(22.5)},1);
     \filldraw[black] (0,0) circle (1.5pt);
     \draw ({-sin(60)},1)--(0,0)--({sin(60)},1);
     \draw (0,0) -- (0,-1);    
\end{tikzpicture}.
\end{center}
\end{itemize}

\begin{example}

Consider the following example of an object $T$ in $\Tree$:

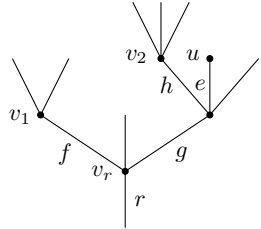
    \begin{center}
\begin{tikzpicture}[scale=0.75]
    \draw (0,1)--(0,0)--(0,-1);
     \filldraw[black] (0,0) circle (1.5pt);
    \draw (-1.5, 1)--(0,0);
    \filldraw[black] (-1.5,1) circle (1.5pt);
    \draw (-2, 2) -- (-1.5, 1) -- (-1, 2);
    \draw (1.5, 1)--(0,0);
    \filldraw[black] (1.5,1) circle (1.5pt);
    \draw ({1.5-sin(60)}, 2)--(1.5, 1)--({1.5+sin(60)}, 2);
    \draw (1.5,2)--(1.5, 1);
    \filldraw[black] (1.5,2) circle (1.5pt);
    \filldraw[black] ({1.5-sin(60)},2) circle (1.5pt);
    \draw ({1.5-sin(60)-sin(30)},3)--({1.5-sin(60)},2)--({1.5-sin(60)+sin(30)},3);
    \draw ({1.5-sin(60)},3)--({1.5-sin(60)},2);
    \node [label={[xshift=0.2cm, yshift=-0.7cm, scale=0.9]$r$}] {};
    \node [label={[xshift=-1.4cm, yshift=0.4cm, scale=0.9]$v_1$}] {};
    \node [label={[xshift=0.15cm, yshift=1.2cm, scale=0.9]$v_2$}] {};
    \node [label={[xshift=0.9cm, yshift=1.25cm, scale=0.9]$u$}] {};
    \node [label={[xshift=-0.8cm, yshift=-0.15cm, scale=0.9]$f$}] {};
    \node [label={[xshift=0.75cm, yshift=-0.1cm, scale=0.9]$g$}] {};
    \node [label={[xshift=1cm, yshift=0.85cm, scale=0.9]$e$}] {};
    \node [label={[xshift=0.55cm, yshift=0.8cm, scale=0.9]$h$}] {};
   \node [label={[xshift=-0.3cm, yshift=-0.35cm, scale=0.9]$v_r$}] {};
\end{tikzpicture}
\end{center}

The root edge corresponds to the edge $r$, whereas the edges $e,f, g$ and $h$ are the inner edges of the tree. The edges which are connected to exactly one vertex are the leaves of $T$. As for the vertices, $v_1$ and $v_2$ are  the leaf vertices of $T$, $u$ is the only stump and $v_r$ is the root vertex.
\label{example}
\end{example}

In a way that mirrors the description of the morphisms in $\mathsf{\Delta}$ via face and degeneracy maps satisfying the simplicial relations, one can show that the morphisms in $\Tree$ are generated under composition by the classes of maps below:
\begin{itemize}[label=$\diamond$]
    \item For every inner edge $e$ of $T$, we have an associated \textit{inner face} $\partial_e T \to T$ which comes from contracting $e$.
    \item For every leaf vertex $v$, there is a \textit{leaf face} $\partial_v T \to T$ coming from removing this vertex from $T$.
    \item Suppose the root vertex of $T$ has the property that one and only one of its incoming edge is also an inner edge. Then $T$ admits a \textit{root face} $\partial_{\mathrm{root}} T \to T$, where the domain is the tree resulting from removing the root vertex of $T$ and keeping the rest of $T$.
    \item If $e$ is any edge of $T$, there is a degeneracy map $\sigma_e T \to T$ which subdivides $e$ in half by adding a new vertex.
    \item The isomorphisms $T \xrightarrow{\cong} S$.
\end{itemize}

We will usually call the faces and degeneracy maps defined above the \textit{elementary faces} and \textit{elementary degeneracies} of a tree.

These morphisms, together with the dendroidal identities in \citep[Sec. 3.3.4]{HeutsMoerdijkDendroidal}, give a complete description of the tree category $\Tree$ in terms of generators and relations. The main motivation for considering this category in the first place is that its presheaves, which form the category $\dSets$ of \textit{dendroidal sets}, correctly model the algebraic structure of an operad. More explicitly, there exists a functor 
$$ N \colon \mathsf{Op} \longrightarrow \dSets,$$
where $\mathsf{Op}$ is the category of coloured operads, which sends a coloured operad $\mathbf{P}$ to a dendroidal set $N\mathbf{P}$, which we will call the \textit{dendroidal nerve} of $\mathbf{P}$ \citep[Ex. 4.2]{MoerdijkWeissDendroidal}, resembling the usual construction of the nerve of a small category. We also point out that there is a pair of adjoint functors
$$\begin{tikzcd}
	{\dSets} & {\sSets}
	\arrow["{\iota^*}"', shift right, from=1-1, to=1-2]
	\arrow["{\iota_!}"', shift right, from=1-2, to=1-1]
\end{tikzcd}$$
comparing simplicial and dendroidal presheaves. This is induced by the inclusion of the full subcategory of $\Tree$ spanned by the trees with only vertices with exactly one incoming edge (which we usually call the \textit{linear trees}), which is isomorphic to $\mathsf{\Delta}$.

Equipped with this new formalism, one can easily transfer some of the notions from the homotopy theory of simplicial sets to the world of dendroidal sets. Let us point out some of these analogous constructions:
\begin{itemize}[label=$\diamond$]
    \item Any tree $T$ admits a \textit{boundary} $\partial T \to T$, given as the union of all the faces of $T$.
    \item Any inner edge $e$ leads to an \textit{inner horn} $\Lambda^e T \to T$, defined as the union of all elementary faces except for $\partial_e T$.
    \item Any leaf vertex $v$ leads to a \textit{leaf horn} $\Lambda^v T \to T$, defined as the union of all elementary faces except for $\partial_v T$.
    \item If $T$ admits a root face, then the \textit{root horn} $\Lambda^{v} T \to T$ corresponds to the union of all faces except the root face. Here $v$ denotes the root vertex.
    \item A dendroidal map $f\colon X \to Y$ is a \textit{normal monomorphism} if it is levelwise injective and, for every tree $T$, the group of automorphisms of $T$ acts freely on $Y_T - \mathrm{im}(f)$.
    \item A dendroidal map is \textit{inner anodyne} if it can be obtained from the inner horn inclusions $\Lambda^e T \to T$ via pushouts, transfinite compositions and retracts (that is, the class of inner anodynes is the saturated class generated by the dendroidal inner horn inclusions). A similar definition applies for \textit{leaf/root anodynes}, which form the saturated class generated by inner horn inclusions and leaf/root horn inclusions.
    \item An \textit{inner fibration} is a dendroidal map $X \to Y$ that has the right lifting property with respect to all inner horn inclusions. A similar definition applies for dendroidal \textit{leaf/right fibrations} in terms of leaf/root anodynes.
\end{itemize}

The dendroidal formalism also allows us to define $\infty$-operads, which we record in the definition below.
\begin{definition}
    Let $X$ be a dendroidal set. We say $X$ is an \textit{$\infty$-operad} if, for any inner horn inclusion $\Lambda^e T \to T$ and solid diagram
\[\begin{tikzcd}[cramped]
	{\Lambda^e T} & X, \\
	T
	\arrow[from=1-1, to=1-2]
	\arrow[from=1-1, to=2-1]
	\arrow[dashed, from=2-1, to=1-2]
\end{tikzcd}\]
    a diagonal dashed lift exists. The \textit{underlying $\infty$-category} of $X$ is the simplicial set $\iota^* X$, which is indeed an $\infty$-category by \citep[Lem. 6.2]{HeutsMoerdijkDendroidal}.
\end{definition}

Before finishing this section on background on dendroidal sets, we will mention some results which we will use in the next section. We recall that given an $\infty$-operad $X$, we say $x \in X_{C_1}$ is an \textit{equivalence} is it defines an equivalence in the underlying $\infty$-category of $X$.

\begin{theorem}
    Let $T$ be a tree with at least two vertices and a unary root vertex $v$ (that is, $v$ has only one incoming edge). Given $p\colon X \to Y$ an inner fibration of $\infty$-operads and a commutative square
    $$
\begin{tikzcd}
\Lambda^v T \arrow[d] \arrow[r, "f"] & X \arrow[d, "p"] \\
T \arrow[r]                     & Y               
\end{tikzcd}
    $$
    such that $f$ sends $v$ to an equivalence in $X$, then the diagram admits a diagonal lift.
    \label{joyal_lifting}
\end{theorem}

This was originally proved by Cisinski--Moerdijk in \citep[Thm. 4.2]{CisinskiMoerdijk}.

In the theory of simplicial sets, a necessary and crucial step in analysing the behaviour of the different classes of anodyne maps is to consider how these interact with the cartesian product in $\sSets$. A similar situation holds for the category of dendroidal sets, but it is also useful in a lot of situation to have a grasp on how these classes of morphisms interact with a different bifunctor
$$ {-} \otimes {-} : \dSets \times \dSets \rightarrow \dSets.$$

This is defined on representables via the formula $T \otimes S = N( \Omega(T) \otimes_{\mathsf{BV}} \Omega(S) )$, where $\otimes_{\mathsf{BV}}$ denotes the Boardman--Vogt tensor product of coloured operads, and then extend to all presheaves by taking left Kan extensions on both variables. The original definition of the Boardman--Vogt tensor product is contained in \cite{BoardmanVogtHoInv} within the more general context of algebraic theories and a similar treatment is contained in Section 1.6 of \cite{HeutsMoerdijkDendroidal} for the specific case of coloured operads.

Let us make two important remarks concerning the dendroidal tensor product. Firstly, $\dSets$ equipped with $\otimes$ does not define a monoidal structure, since the tensor product will not be associative, even at the level of the representable presheaves. Secondly, in general there are no projection maps associated to the dendroidal tensor product; however, in the special case when $X \in \sSets$ and $Y \in \dSets$, there exists a projection $\pi_Y \colon X \otimes Y \to Y$ onto the dendroidal component, see Section 4.2 of \cite{HeutsMoerdijkDendroidal} for details.

The dendroidal tensor product takes part in an adjunction, which, for any $X, Y, A \in \dSets$, takes the form 
$$ \dSets(X \otimes A, Y) \cong \dSets(X, \mathsf{Hom}(A,Y)).$$ 

We will reserve the notation $Y^A \in \sSets$ for the underlying simplicial set  of the dendroidal set $\mathsf{Hom}(A,Y)$. As one would expect and immediately checks, the 0-simplices of $Y^A$ are the dendroidal maps $ A \to Y$.

For the next few statements, we will represent the leaf edge of $C_1$ by $0$, and the root edge by $1$. We also recall that, for a dendroidal map $p \colon X \to Y$ and colour $y \in Y_\eta$, we can define the \textit{fiber of $y$} via the pullback
\[\begin{tikzcd}
	{X_y} & X \\
	\eta & Y.
	\arrow[from=1-1, to=1-2]
	\arrow[from=1-1, to=2-1]
	\arrow["\lrcorner"{anchor=center, pos=0, scale=1.5}, draw=none, from=1-1, to=2-2]
	\arrow["p", from=1-2, to=2-2]
	\arrow["y", from=2-1, to=2-2]
\end{tikzcd}\]
It follows from this that $X_y$ is actually a simplicial set due to the isomorphism $\dSets_{/\eta} \cong \sSets$: Moreover, it is an $\infty$-category if $p$ is a dendroidal inner fibration.

\begin{lemma}
    Let $i \colon A \to B$ be a dendroidal inner/leaf/root anodyne and $p \colon X \to Y$ a dendroidal inner/left/right fibration. Then the canonical map to the pullback
    \begin{equation*}
        \Psi : X^B \longrightarrow X^A \times_{Y^A} Y^B
    \end{equation*}
    is a trivial Kan fibration. 
    
    Moreover, suppose there are dendroidal maps $f,g\colon B \to X$ projecting to the same pair $(\alpha, \beta)$ via the map $\Psi$. Then there exists a dendroidal morphism $H\colon C_1 \otimes B \to X$ with the following properties:
        \begin{enumerate}[label=(\alph*)]
            \item $H(0, -) = g$ and $H(1,-) = f$.
    \item The following diagrams
    $$
\begin{tikzcd}
C_1 \otimes B \arrow[rr, "H"] \arrow[d, "\pi_B"'] &  & X \arrow[d, "p"] & C_1 \otimes A \arrow[rr, "\mathrm{id} \otimes i"] \arrow[d, "\pi_A"] &  & C_1 \otimes B \arrow[d, "H"] \\
B \arrow[rr, "\beta"]                             &  & Y                & A \arrow[rr, "\alpha"]                                               &  & X                           
\end{tikzcd}
    $$
    commute.

    \item For every colour $b \in B_{\eta}$, the 1-corolla $H(-,b) \colon C_1 \to X$ defines an equivalence in the fiber $X_{g(b)}$ of $p$.
        \end{enumerate}
        
    \label{J_homotopy}
\end{lemma}

\begin{proof}
    The proof of the first statement is contained \citep[Thm. 6.33]{HeutsMoerdijkDendroidal}. As for the second one, we can more succinctly describe the map $H$ as a diagonal lift in the diagram
\[\begin{tikzcd}
\partial J \otimes B \cup J \otimes A \arrow[d] \arrow[r] & X \arrow[d, "p"] \\
J \otimes B \arrow[r] \arrow[ru, dashed]                  & Y               
\end{tikzcd}\]
    where the top map is $\left( g \amalg f, \alpha \pi_A\right)$ and the bottom one is $\beta \pi_B$, and $J$ is nerve of the groupoid with two objects and an isomorphism between them. By adjunction, this is equivalent to finding a lift in 
    $$
\begin{tikzcd}
\partial J \arrow[d] \arrow[r] & X^B \arrow[d, "\Psi"]   \\
J \arrow[r] \arrow[ru, dashed] & X^A \times_{Y^A} Y^B.
\end{tikzcd}
    $$

    By the first part and the fact that $\partial J \to J$ is an anodyne map, we conclude that the desired lift exists.
    \end{proof}

\begin{remark}
    The map $H$ is sometimes known as a \textit{fibrewise $J$-homotopy} relative to $A$ between $f$ and $g$. For a more thorough study of this notion, see Section 6.8 of \cite{HeutsMoerdijkDendroidal}.
\end{remark}

\section{The right cancellation property for classes of dendroidal morphisms}

In order to prove Theorem \ref{right_cancel_dendroidal} and following Stevenson's own proof, we start by showing the following auxiliary lemma. 

\begin{lemma}
    Let $p\colon X \to Y$ be an inner fibration of dendroidal sets and $j\colon B \to C$ a normal monomorphism. Consider any commutative diagram
    \begin{equation}
\begin{tikzcd}
0 \otimes C \cup_{0 \otimes B} C_1 \otimes B \arrow[rr, "K"] \arrow[d] &                  & X \arrow[d, "p"] \\
C_1 \otimes C \arrow[r, "\pi_C"]                                       & C \arrow[r, "g"] & Y               
\end{tikzcd}
\label{aux_lemma_lift}
    \end{equation}
such that, for every colour $b \in B_{\eta}$, the 1-corolla $K(-, b) \colon C_1 \to X$ is an equivalence in the fiber $X_{g(b)}$ of $p$. Then a diagonal lift always exists.
\label{technical_lift_joyal}
\end{lemma}

\begin{proof}
    Firstly, we can assume that $j$ is a boundary inclusion $\partial T \to T$, due to the existence of a skeletal filtration for normal monomorphisms \citep[Prop. 3.26]{HeutsMoerdijkDendroidal}. Moreover, by pulling back $p$ along $g$ and using that any tree $T$ defines an $\infty$-operad, we can further assume that $X$ and $Y$ are $\infty$-operads.

    According to Theorem B.2 in \cite{CisinskiMoerdijk}, there is a finite filtration of the dendroidal set $C_1 \otimes T$
    $$A_0 \subseteq A_1 \subseteq \cdots \subseteq A_{n-1} \subseteq A_n = C_1 \otimes T,$$
    starting at $A_0 = 0 \otimes T \cup_{0 \otimes \partial T} C_1 \otimes \partial T$ and satisfying the following properties:
    \begin{enumerate}[label=(\roman*)]
        \item The inclusion $A_j \to A_{j+1}$ is inner anodyne, as along as $j \neq n-1$.
        \item There exists a pushout square
\[\begin{tikzcd}[cramped]
	{\Lambda^v S} & {A_{n-1}} \\
	S & {A_n}
	\arrow[from=1-1, to=1-2]
	\arrow[from=1-1, to=2-1]
	\arrow[from=1-2, to=2-2]
	\arrow[from=2-1, to=2-2]
	\arrow["\lrcorner"{anchor=center, pos=0, rotate=180, scale=1.5}, draw=none, from=2-2, to=1-1]
\end{tikzcd}\]
        where $S$ is a tree with at least two vertices and a unary root vertex $v$. Moreover, the inclusion of the vertex $v\colon C_1 \to C_1 \otimes T$ coincides with the inclusion
        $ C_1 \otimes r_T \to C_1 \otimes T ,$
        where $r_T$ denotes the root edge of $T$.
    \end{enumerate}

    Condition (i), together with $p$ being an inner fibration, provides us with a diagonal lift for diagram \eqref{aux_lemma_lift} up to stage $A_{n-1}$ of the filtration. For the last step of the filtration, consider the diagram
\[\begin{tikzcd}[cramped]
	{\Lambda^v S} & {A_{n-1}} & X \\
	S & { C_1 \otimes T} & Y
	\arrow[from=1-1, to=1-2]
	\arrow[from=1-1, to=2-1]
	\arrow[from=1-2, to=1-3]
	\arrow[from=1-2, to=2-2]
	\arrow["p", from=1-3, to=2-3]
	\arrow[from=2-1, to=2-2]
	\arrow["\lrcorner"{anchor=center, pos=0, rotate=180, scale=1.5}, draw=none, from=2-2, to=1-1]
	\arrow["{g \pi_T}", from=2-2, to=2-3]
\end{tikzcd}\]

    It suffices to show that the outer square admits a diagonal lift, since the leftmost square is a pushout diagram. With this purpose in mind, observe that property (ii) and our initial hypothesis on $K$ imply that the top map in the diagram sends the unitary vertex $v$ to an equivalence in $X$. Thus, the conditions of Theorem \ref{joyal_lifting} are satisfied, which provides us with our desired diagonal lift.
\end{proof}

\begin{proof}[Proof of Theorem \ref{right_cancel_dendroidal}]
    We will first consider the case of inner anodynes and in the end discuss how the proof goes for leaf and root anodynes.
    
    Suppose we are given the diagram below
\begin{equation*}
\begin{tikzcd}
A \arrow[d, "i"'] \arrow[rd, "fi"] &                  \\
B \arrow[d, "j"'] \arrow[r, "f"]   & X \arrow[d, "p"] \\
C \arrow[r, "g"]                   & Y            
\end{tikzcd},
\end{equation*}
    with $p$ being an inner fibration. In order to find a diagonal lift in the bottom square, note first that $ji$ being inner anodyne provides us with a morphism $s\colon C \to X$ making the outer square commute; however, this will not make the bottom square commute in general since the equation $sj = f$ may not hold.
    
    Since $i$ is inner anodyne and $p$ is an inner fibration, the canonical map
    \begin{equation*}
    \Psi : X^B \longrightarrow X^A \times_{Y^A} Y^B,
    \end{equation*}
    is a trivial Kan fibration by Lemma \ref{J_homotopy}. Moreover, by the construction of $s$, both of the morphisms $f, sj\colon B \to X$ lie on the fiber of $\Psi$ over the pair $(fi, gj)$, and therefore the second part of Lemma \ref{J_homotopy} can be applied. Spelling out the details, this entails the existence of a dendroidal morphism $H\colon C_1 \otimes B \to X$ such that the following hold:
   \begin{enumerate}[label=(\alph*)]
    \item $H(0, -) = sj$ and $H(1, -) = f$.
    \item The diagrams
    $$
\begin{tikzcd}
C_1 \otimes B \arrow[d, "\pi_B"'] \arrow[rr, "H"] &                  & X \arrow[d, "p"] & C_1 \otimes A \arrow[d, "\pi_A"'] \arrow[rr, "C_1 \otimes i"] &                  & C_1 \otimes B \arrow[d, "H"] \\
B \arrow[r, "j"]                                & C \arrow[r, "g"] & Y                & A \arrow[r, "i"]                                          & B \arrow[r, "f"] & X                         
\end{tikzcd}
    $$
    commute.

    \item For any colour $b \in B_{\eta}$, the 1-corolla $H(-,b)\colon C_1 \to X$ defines an equivalence in the fiber $X_{g(b)}$ of $p$.
   \end{enumerate}

   The homotopy $H\colon C_1 \otimes B \to X$ and the diagonal lift $s\colon C \cong 0\otimes C \to X$ give rise to a map $(s,H)\colon 0 \otimes C \cup_{0 \otimes B} C_1 \otimes B \to X$ on the respective pushout, which additionally makes the following diagram
\begin{equation}
\begin{tikzcd}
0 \otimes C \cup_{0 \otimes B} C_1 \otimes B \arrow[rr, "{(s,H)}"] \arrow[d] &                  & X \arrow[d, "p"] \\
C_1 \otimes C \arrow[r, "\pi_C"]                                             & C \arrow[r, "g"] & Y               
\end{tikzcd}
\label{diagram_extra}
\end{equation}
commute. By property (c) above, the restrictions $(s,H)(-,b)$ for each colour $b \in B_{\eta}$ define an equivalence in the fiber $X_{g(b)}$ of $p$. Therefore the conditions for Lemma \ref{technical_lift_joyal} are satisfied and a diagonal lift $t\colon C_1 \otimes C \to X$ can be constructed. It is now an easy verification that $t(1, -)$ provides the required lift for our initial diagram, as desired.

Finally, we discuss the leaf and root anodyne cases. Firstly, these classes of anodynes have their respective versions of Lemma \ref{J_homotopy}, therefore the only problem rests in Lemma \ref{technical_lift_joyal}. Of course, since left and right fibrations are particular cases of inner fibrations, this result still holds; but notice that we can show that \eqref{diagram_extra} admits a lift without this extra technical lemma. Indeed, due to the property that $(s,H)$ define certain equivalences in the fibers, we can actually replace the left map in \eqref{diagram_extra} by
$$ 0 \otimes C \cup_{0 \otimes B} J \otimes B \longrightarrow J \otimes C.$$
Since $0 \to J$ is both a left and right anodyne, an application of \citep[Cor. 6.30]{HeutsMoerdijkDendroidal} shows that this pushout-product is a leaf and root anodyne.
\end{proof}

\begin{remark}
One can deduce Stevenson's original simplicial version of Theorem \ref{right_cancel_dendroidal} from the dendroidal variation we just presented. Indeed, if $A \xrightarrow{i} B \xrightarrow{j} C$ is a string of simplicial monomorphisms with $ji$ and $i$ inner anodynes, then applying $\iota_!$ to it yields a string of dendroidal sets
$$ \iota_! A \xlongrightarrow{\iota_! i} \iota_! B \xlongrightarrow{\iota_! j} \iota_! C.$$
All the monomorphisms in questions are normal (being normal is meaningless in this context), and the dendroidal maps $\iota_! i$ and $\iota_!(ji)$ are inner anodynes. Therefore, by right cancellation so is $\iota_! j$, which is equivalent to saying $j$ is a simplicial inner anodyne. This argument also holds for the other two classes of anodynes we discussed.
\end{remark}

\section{An application: the symmetric monoidal envelope of a dendroidal $\mathbf{\infty}$-operad}

In this section we will give an application of the right cancellation property for dendroidal inner anodynes in order to explain how one can associate to any dendroidal set $X$ a simplicial set $\Proptensor(X)^{\otimes}$ which, when applied to the dendroidal nerve of a coloured operad $X = N\mathbf{P}$, should recover the symmetric monoidal envelope of the operad, which we recall below. Here we will write $\underline{n} = \{ 1, 2, \ldots, n\}$, which we often see also as the poset $\{1 < 2 < \cdots < n\}$.
\begin{definition}
    Let $\mathbf{P}$ be a coloured operad. We define the \textit{symmetric monoidal envelope of $\mathbf{P}$} is the symmetric monoidal category $(\Prop(\mathbf{P}), \boxplus, \emptyset)$ described by the following:
    \begin{itemize}[label=$\diamond$]
        \item The objects are the strings $(c_1, \ldots, c_n)$ of colours of $\mathbf{P}$, for $n \geq 0$. The empty string is respresented by $\emptyset$.
        \item A morphism
        $$ (c_1, \ldots, c_n) \longrightarrow (d_1, \ldots, d_m)$$
        is determined by a pair $(f, \{ p_j : j \in \underline{m} \})$, where $f\colon \underline{n} \to \underline{m}$ is a function of finite sets. For each $j \in \underline{m}$, $p_j$ is an operation in $\mathbf{P}$ with output $d_j$, and input set $\{ c_i : f(i) = j \} \subseteq \{ c_1, \ldots, c_n\}$ ordered according to $\underline{n}$. Composition comes from the operadic composition of $\mathbf{P}$.
        \item The tensor product is given on objects via concatenation
        $$ (c_1, \ldots, c_n) \boxplus (d_1, \ldots, d_m) = (c_1, \ldots, c_n, d_1, \dots, d_m),$$
        with unit given by the empty tuple.
    \end{itemize}
    \label{Prop_Strict}
\end{definition}

\begin{remark}
    The symmetric monoidal category $\mathsf{Env}(\mathbf{P})$ has appeared under many different names throughout the years. They initially appeared under the name of PROP in \cite{MacLaneNatural}, connected to the study of cohomology operations by Adams and Mac Lane. In this work we are working with the notion of a \textit{coloured} PROP, which is the one also appearing in the work of Boardman and Vogt in \cite{BoardmanVogtHoInv}, but not the one used by Hackney and Robertson in \cite{HackneyRobertsonProps}. 
    
    For a great survey on the history of PROPs see \cite{MarklProps}, and for more on coloured operads \cite{HackneyRobertsonProps} is a good resource.
\end{remark}
\subsection{Interpreting $\Fin_\ast$ as a category of forests}

\leavevmode

We will begin by defining a certain category of forests $\mathbb{F}$ which will be useful in keeping track of the combinatorics of $\Prop(\mathbf{P})$. By a \textit{forest} we mean a finite (possibly empty) collection of trees $\{T_i \in \Tree : i = 1, \dots, n \}$, which we will usually denote by $\bigoplus_{i=1}^n T_i$, and in the special case where all the $T_i$'s are the same tree $T$, we abbreviate the notation to $n \cdot T$. We note that we can view any forest $\bigoplus_{i=1}^n T_i$ as the coproduct of representables $\coprod_{i=1}^n T_i$ in $\dSets$, which we be useful at times. We also allow for the possibility of an empty forest, which we write $\emptyset$.

The category $\mathbb{F}$ is actually isomorphic to the category of finite pointed sets $\Fin_\ast$, but it is useful to have this formulation of the latter category in terms of forests for the subsequent results. This is done as follows:

\begin{itemize}[label=$\diamond$]
    \item A pointed set $\langle n \rangle = \{ 1, \ldots, n \} \cup \{ \ast \}$ will be represented by the forest $n \cdot \eta$, which is $\emptyset$ whenever $n = 0$.
    \item A function of pointed sets $\alpha \colon \langle m \rangle \to \langle n \rangle$ is given by a forest  
    $$F = \left( \bigoplus_{j=1}^n C_{\lvert \alpha^{-1}(j) \rvert} \right) \oplus \lvert \alpha^{-1}(\ast) \rvert \cdot \eta,$$
    together with an injective function $\mathsf{input}\colon \underline{m} \to E(F)$ with image the edges of the corollas and the copies of $\eta$, and satisfying the relation that the set $\mathsf{input}(\alpha^{-1}(j))$ is the leaf set of $C_{\lvert \alpha^{-1}(j) \rvert}$, for each $1 \leq j \leq n$.
    
  We will graphically represent the information of $\alpha$ in terms of a forest of the form
    
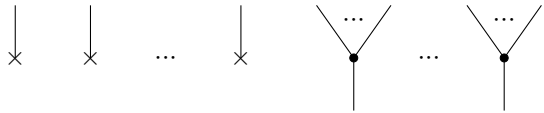
\vspace{1em}

    \begin{center}
    \begin{tikzpicture}
    \draw (0,0.7)--(0,0);
    \draw (0,0) node {$\times$};
    \draw (1,0.7)--(1,0);
    \draw (1,0) node {$\times$};
    \draw (1.9,0) node {$\cdot$};
    \draw (2,0) node {$\cdot$};
    \draw (2.1,0) node {$\cdot$};
    \draw (3,0.7)--(3,0);
    \draw (3,0) node {$\times$};
    \draw (4.5,0)--(4.5,-0.7);
    \draw (-{0.7*sin(45)+4.5},0.7) -- (4.5,0) -- ({0.7*sin(45)+4.5},0.7);
     \draw (4.4,0.5) node {$\cdot$};
    \draw (4.5,0.5) node {$\cdot$};
    \draw (4.6,0.5) node {$\cdot$};
     \filldraw[black] (4.5,0) circle (1.5pt);
     \draw (5.4,0) node {$\cdot$};
    \draw (5.5,0) node {$\cdot$};
    \draw (5.6,0) node {$\cdot$};
    \draw (6.5,0)--(6.5,-0.7);
    \draw (-{0.7*sin(45)+6.5},0.7) -- (6.5,0) -- ({0.7*sin(45)+6.5},0.7);
    \draw (6.4,0.5) node {$\cdot$};
    \draw (6.5,0.5) node {$\cdot$};
    \draw (6.6,0.5) node {$\cdot$};
    \filldraw[black] (6.5,0) circle (1.5pt);
    
    \end{tikzpicture}
    \end{center}
    where the edges with the symbol $\times$ correspond to the copies of $\eta$, and the function $\mathsf{input}$ labels the leaves of the corollas. The target map is given by keeping all the edges at the level of the roots, and the source by keeping all the edges at the level of the leaves, which includes the edges above the symbol $\times$.

    \item Given composable morphisms 
    $$\ell \cdot \eta \xlongrightarrow{(F_1, \mathsf{input}_1)} m \cdot \eta \xlongrightarrow{(F_2, \mathsf{input}_2)} n \cdot \eta$$
    the composition is given by the pair $(G, \mathsf{input}_1)$. Here $G$ is obtained by first grafting $F_1$ along the leaves of $F_2$  (we attach the $i^{th}$ copy of $\eta$ in $m \cdot \eta$ to the leaf of $F_2$ with the label $\mathsf{input}_2(i)$) and then contracting the edges that are identified.
\end{itemize}

More generally, a $k$-simplex  in $\Fin_\ast$
$$\langle n_0 \rangle \xrightarrow{\alpha_1} \langle n_1 \rangle \xrightarrow{\alpha_2} \cdots \xrightarrow{\alpha_{k-1}} \langle n_{k-1} \rangle \xrightarrow{\alpha_k} \langle n_k \rangle,$$
will be given by a string
$$n_0 \cdot  \eta \xrightarrow{(F_1, \mathsf{input}_1)} n_1 \cdot \eta \longrightarrow \cdots \longrightarrow n_{k-1} \cdot \eta \xrightarrow{(F_k, \mathsf{input}_k)} n_k \cdot \eta,$$
which is encoded by the triple $(F, \mathsf{input}_1)$, where $F$ is the forest obtained via grafting the $F_i$'s in a similar fashion to what was done when describing the composition, and $\mathsf{input}_1$ labels the leaves of $F$. This reformulation defines a category isomorphic to $\Fin_\ast$, which we will denote by $\mathbb{F}$.

\begin{example} Consider the 3-simplex in $\mathbb{F}$
$$5 \cdot \eta \xrightarrow{F_1} 4 \cdot \eta \xrightarrow{F_2} 3 \cdot \eta \xrightarrow{F_3} \eta,$$
pictorially represented by the forests

 
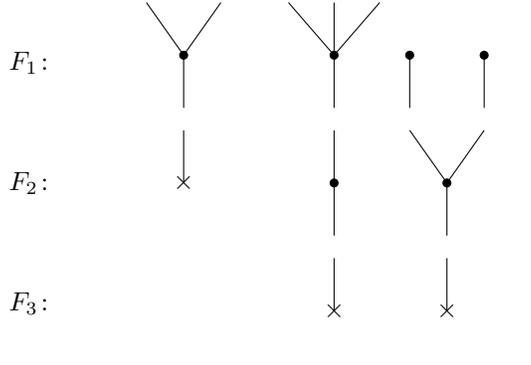
\vspace{1em}

    \begin{center}
    \begin{tikzpicture}
    \draw (0,0)--(0,-0.7);
     \filldraw[black] (0,0) circle (1.5pt);
     \draw (-{0.7*sin(45)},0.7) -- (0,0) -- ({0.7*sin(45)},0.7);
     \draw (2,0.7) -- (2,0) -- (2,-0.7);
     \filldraw[black] (2,0) circle (1.5pt);
     \draw (-{0.7*sin(60)+2},0.7) -- (2,0) -- ({0.7*sin(60)+2},0.7);
      \draw (-{0.7*sin(45)+3.5},0) -- (-{0.7*sin(45)+3.5},-0.7);
     \filldraw[black] (-{0.7*sin(45)+3.5},0) circle (1.5pt);
     \draw ({0.7*sin(45)+3.5},0) -- ({0.7*sin(45)+3.5},-0.7);
     \filldraw[black] ({0.7*sin(45)+3.5},0) circle (1.5pt);
     \node [label={[xshift=-2cm, yshift=-0.5cm]$F_1\colon$}] {};

    \draw (0,-1) -- (0,-1.7);
     \draw (0,-1.7) node {$\times$};
     \draw (2,-1) -- (2,-1.7) -- (2,-2.4);
     \filldraw[black] (2,-1.7) circle (1.5pt);
      \draw (3.5,-1.7) -- (3.5,-2.4);
     \filldraw[black] (3.5,-1.7) circle (1.5pt);
      \draw (-{0.7*sin(45)+3.5},-1) -- (3.5,-1.7) -- ({0.7*sin(45)+3.5},-1);
     \draw (4.7,-1.7) -- (4.7,-2.5);
     \filldraw[black] (4.7,-1.7) circle (1.5pt);
     \node [label={[xshift=-2cm, yshift=-2.1cm]$F_2\colon$}] {};

     \draw (2,-2.7) -- (2,-3.4);
     \draw (2,-3.4) node {$\times$};
     \draw (3.5,-2.7) -- (3.5,-3.4);
     \draw (3.5,-3.4) node {$\times$};
     \draw (4.7,-2.7) -- (4.7,-3.4) -- (4.7,-4.1);
     \filldraw[black] (4.7,-3.4) circle (1.5pt);
     \node [label={[xshift=-2cm, yshift=-3.7cm]$F_3\colon$}] {};
\end{tikzpicture}
\end{center}

For the sake of simplicity, we have omitted the labelling for the inputs and outputs, but it should be understood that they are labelled increasingly from left to right. The forest $F$ associated to such a 3-simplex is the forest $F$ below

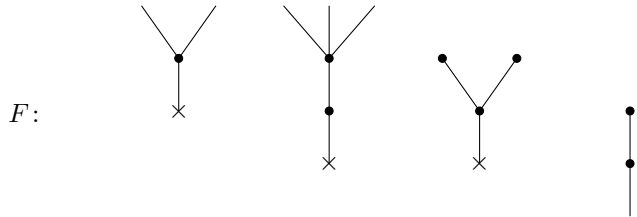
\vspace{1em}

\begin{center}
    \begin{tikzpicture}
    \draw (0,1.4) -- (0,0.7);
    \draw (-{0.7*sin(45)},2.1) -- (0,1.4) -- ({0.7*sin(45)},2.1);
    \draw (0,0.7) node {$\times$};
     \filldraw[black] (0,1.4) circle (1.5pt);
     \draw (2,2.1) -- (2,1.4) -- (2,0.7);
    \draw (-{0.7*sin(60)+2},2.1) -- (2,1.4) -- ({0.7*sin(60)+2},2.1);
     \filldraw[black] (2,0.7) circle (1.5pt);
     \filldraw[black] (2,1.4) circle (1.5pt);
     \draw (2,0.7) -- (2,0);
     \draw (2,0) node {$\times$};
    \draw (-{0.7*sin(45)+4},1.4) -- (4,0.7) -- ({0.7*sin(45)+4},1.4);
     \filldraw[black] (4,0.7) circle (1.5pt);
     \filldraw[black] (-{0.7*sin(45)+4},1.4) circle (1.5pt);
     \filldraw[black] ({0.7*sin(45)+4},1.4) circle (1.5pt);
     \draw (4,0.7)--(4,0);
     \draw (4,0) node {$\times$};
     \draw (6,0.7) -- (6,0)--(6,-0.7);
     \filldraw[black] (6,0) circle (1.5pt);
     \filldraw[black] (6,0.7) circle (1.5pt);
     \node [label={[xshift=-2cm, yshift=0.3cm]$F\colon$}] {};
\end{tikzpicture}
\end{center}

One can also describe the action of the simplicial face maps $\partial_i \colon N\mathbb{F}_3 \to N\mathbb{F}_2$ on the 3-simplex in question in terms of the forest $F$ (here we will again ignore the action of the face maps on the labelling data):
\begin{itemize}[label=$\diamond$]
    \item $\partial_0 F$ is obtained by removing the $\times$ symbol from trees 2 and 3, and applying the root face to tree 4.
    \item $\partial_1 F$ is obtained by removing the $\times$ symbol from tree 1, applying the root face to trees 2 and 3, and contracting $w$ in tree 4. When the root face is applied to trees 2 and 3, the symbol $\times$ propagates to the newly-obtained trees.
    \item $\partial_2 F$ is obtained by applying the root face to tree 1, contracting $x, y$ and $z$ in trees 2 and 3, and removing the stump in tree 4.
    \item $\partial_3 F$ is obtained by applying the leaf faces in trees 1, 2 and 3.
\end{itemize}

It is not difficult to deduce from this the general pattern of the face maps on a generic element of $N\mathbb{F}_k$, by removing and contracting certain "layers" of the associated forest. 
\end{example}

From now on we will ignore the information on a $k$-simplex of $N\mathbb{F}$ given by the labelling, and instead think of such a simplex as a certain forest $F$. In particular, this means that we can see each simplex in $N \mathbb{F}$ as a dendroidal set, as we explained at the start of the section. Before delving into the technical results we will need, we will present some useful definitions.
\begin{definition}
    Let $k \geq 1$ and $F$ a $k$-simplex in $N\mathbb{F}$.
    \begin{enumerate}[label=(\alph*)]
        \item We say a component tree $T$ of $F$ is \textit{uprooted} if its graphical representation corresponds to a tree containing the symbol $\times$.
        \item We write $\drl F \subseteq F$ for the dendroidal subset given by the union of the faces $\partial_0 F \cup \partial_k F \subseteq F$.
        \item For $0 \leq j \leq n$, we write $\Lambda^j F \subseteq F$ for the dendroidal set given by the union of all faces $\partial_i F$ such that $i \neq j$.  
    \end{enumerate}
    \label{definition_horns}
\end{definition}

The next two lemmas will be the main technical results that will be used in the next section. We point out that here is where we will make heavy use of the right cancellation property for inner anodynes, which lets us avoid complicated combinatorial arguments about the dendroidal category.

\begin{lemma}
    Let $k \geq 2$ and $F \in N\mathbb{F}_k$. Then the dendroidal inclusion $\drl F \to F$ is inner anodyne.
    \label{rl_inner}
\end{lemma}

\vspace{-2em}

\begin{proof}
    Let us start by making some simplifying assumptions on $F$:
    
    \begin{itemize}[label=$\diamond$]
        \item We can assume that $F$ is a tree $T$ since the action of the face maps $\partial_0$ and $\partial_k$ can be restricted to each component of $F$.
        \item We can further assume that $T$ is not uprooted: we would then have the equality $\partial_0 T = T$ as dendroidal sets and therefore $\drl T= T$.
        \item If $k=2$ then $\drl T \to T$ is the spine inclusion of $T$, which is always inner anodyne by \citep[Lem. 6.37]{HeutsMoerdijkDendroidal}. Therefore we may additionally assume $k \geq 3$.
    \end{itemize}

   Before discussing the general case, we will first exemplify what happens when $k=3$: writing $T$ as a grafting $C_n \circ (T_1, \ldots, T_n)$ for $n \geq 0$, then the dendroidal set $\partial_0 F$ is $T_1 \cup \cdots \cup T_k$, and therefore the inclusion $\partial_{\mathrm{rl}} T \to T$ will correspond to
    $$ j \colon T_1 \cup \cdots \cup T_n \cup \partial_3 T \longrightarrow T,$$
    where we note that the trees $T_1, \ldots, T_n$ and $\partial_3 T$ are elements of $N\mathbb{F}_2$. We can now apply $\partial_{\mathrm{rl}}$ to each of the components of the domain of $j$, leading to the string of composable maps
    $$ \partial_{\mathrm{rl}}^2 T := \partial_{\mathrm{rl}} T_1 \cup \cdots \cup \partial_{\mathrm{rl}} T_n \cup  \partial_{\mathrm{rl}} (\partial_3 T) \xlongrightarrow{i} T_1 \cup \cdots \cup T_n \cup \partial_3 T \xlongrightarrow{j} T.$$

    The dendroidal inclusion $ji\colon \partial_{\mathrm{rl}}^2 (T) \to T$ is inner anodyne, since it coincides with the spine inclusion of $T$. Moreover, $i$ is also inner anodyne by the induction hypothesis. We can now apply the right cancellation property for inner anodynes to conclude that $j$ is also inner anodyne, as we wanted to show.

    For the general case, we first note that we can extend the previous definition of $\partial_{\mathrm{rl}}^2 T$ to arbitrary higher powers $\partial_{\mathrm{rl}}^n T$ by inductively applying the operator $\partial_{\mathrm{rl}}$ to each component of $\partial_{\mathrm{rl}}^{n-1} T$. For a general $T \in N\mathbb{F}_k$, we consider now the dendroidal inclusions 
    $$ \partial_{\mathrm{rl}}^{k-1} T \longrightarrow \partial_{\mathrm{rl}} T \longrightarrow T.$$
    We see that the composite is again just the spine inclusion, and $ \partial_{\mathrm{rl}}^{k-1} T \rightarrow \partial_{\mathrm{rl}} T$ is inner anodyne by the induction hypothesis. Consequently, another application of Theorem \ref{right_cancel_dendroidal} finishes the proof.
\end{proof}

\begin{lemma}
    Let $k \geq 2$ and $F \in N\mathbb{F}_k$. Then, for each $0 < j < k$, the horn inclusion $\Lambda^j F \to F$ is dendroidal inner anodyne.
    \label{horn_inner}
\end{lemma}

\begin{proof}
    The case when $k=2$ is exactly Lemma \ref{rl_inner}, so we can assume that $k \geq 3$ from now on. Furthermore, by the type of arguments as at the start of the previous lemma, we can additionally assume that $F$ has no uprooted components.
    
    For $0 \leq \ell \leq k-1$, set $\partial_{\leq \ell} F \subseteq \Lambda^j F$ to be the dendroidal subset obtained from $\drl F$ by attaching the faces $\partial_i F$, with $ 0 \leq i \leq \ell$ and $i \neq j$, which leads to a filtration
$$ \drl F = \partial_{\leq 0} F \subseteq \partial_{\leq 1} F \subseteq \cdots \subseteq \partial_{\leq k-1} F = \Lambda^j F.$$
   
   It suffices to show that each stage of this filtration is inner anodyne: indeed, the previous lemma together with the right cancellation property for inner anodynes applied to $\drl F \to \Lambda^j F \to F$ will then finish the proof.
   
   The first stage of the filtration can be represented via the following pushout square in $\dSets$:
\[\begin{tikzcd}
	{\drl \partial_1 F} & {\drl F} \\
	{\partial_1 F} & {\partial_{\leq 1}F}
	\arrow[from=1-1, to=1-2]
	\arrow[from=1-1, to=2-1]
	\arrow[from=1-2, to=2-2]
	\arrow[from=2-1, to=2-2]
	\arrow["\lrcorner"{anchor=center, pos=0, rotate=180, scale=1.5}, draw=none, from=2-2, to=1-1]
\end{tikzcd}\]
    
    The left map is inner anodyne by Lemma \ref{rl_inner} and therefore so is $\drl F \to \partial_{\leq 1} F$. For the higher stages of the filtration we have instead the description below:
    \[\begin{tikzcd}
	{\partial_{\leq i-1} \left( \partial_i F \right)} & {\partial_{\leq i-1} F} \\
	{\partial_i F} & {\partial_{\leq i}F}
	\arrow[from=1-1, to=1-2]
	\arrow[from=1-1, to=2-1]
	\arrow[from=1-2, to=2-2]
	\arrow[from=2-1, to=2-2]
	\arrow["\lrcorner"{anchor=center, pos=0, rotate=180, scale=1.5}, draw=none, from=2-2, to=1-1]
\end{tikzcd}\]

    Since $\partial_i F \in N\mathbb{F}_{k-1}$, we conclude that $\partial_{\leq i-1} \left( \partial_i F \right) \to \partial_i F$ is inner anodyne, by the induction hypothesis on both the height $k$ of the forest and on the stage of the filtration $i$. By the closure of inner anodynes under pushouts, the same holds for $\partial_{\leq i-1} F  \to \partial_{ \leq i} F$, as we wanted to show.
\end{proof}

\subsection{The symmetric monoidal envelope of a dendroidal $\infty$-operad}

\leavevmode

The reason for introducing and discussing the category $\mathbb{F}$ in the previous section is that it will help us encode the symmetric monoidal structure of $\Prop(\mathbf{P})$. However, this will not be enough for our purposes, since we also need to keep track of the colours of our operad. Instead, we will consider a slightly more elaborate version of $\mathbb{F}$, which we introduce now.

\begin{definition}
    The category of \textit{decorated forests} $\Fdec$ is the category given by the following data:
    \begin{itemize}[label=$\diamond$]
        \item An object is given by an $n$-tuple $(M_1, \ldots, M_n)$ for $n \geq 0$, where each $M_i$ is the forest  $k_i \cdot \eta$ for some integer $k_i \geq 0$. Such an $n$-tuple partitions the edges of the forest $\bigoplus_{i=1}^n M_i$. 
        \item A morphism in $\Fdec$
        $$ (L_1, \ldots, L_{\ell}) \xlongrightarrow{(\beta,F)} (M_1, \ldots, M_m)$$
       is given by a function of finite pointed sets $\beta\colon \langle \ell \rangle \to \langle m \rangle$, together with a forest $F$ which is a morphism of $\mathbb{F}$. Additionally, these must interact with each other in the following way:
       \begin{enumerate}[label=(\alph*)]
           \item The root edges of $F$ are partitioned according to the $M_i$'s, and all the remaining edges of $F$ (these are the leaves of the corollas and the edges of the uprooted components) are partitioned according to the $L_j$'s.
           \item The set of leaves over the edges in $M_i$ is the union of the edges contained in the $L_j$'s satisfying $\beta(j)=i$.
       \end{enumerate}

        The identity morphism will correspond to seeting $\beta$ to be the corresponding identity function, and $F$ is a forest of only 1-corollas.
       
        \item The composition of the morphisms
        $$ (L_1, \ldots, L_\ell) \xlongrightarrow{(\beta_1,F_1)} (M_1, \ldots, M_m) \xlongrightarrow{(\beta_2, F_2)} (N_1, \ldots, N_n)$$
        is the pair $(\beta_2 \beta_1, F_2 F_1)$, where $F_2 F_1$ comes from the composition in $\mathbb{F}$, see Remark \ref{labellings} for a clarification on how the grafting is done.
    \end{itemize}
\end{definition}

\begin{remark}
    As we have done before, when defining the morphisms in $\Fdec$ we have omitted the information regarding the function $\mathsf{input}$ that labels the inputs of $F$. However, the set up of $\Fdec$ gives a canonical choice of such a labelling: indeed, the leaves of $F$ form the forest of edges $\bigoplus_{i=1}^\ell L_i$, which is ordered via the internal ordering of each $L_i$, together with the usual order of the indexing set $\{1, \ldots, \ell \}$. This is the labelling which is assumed when defining the composition in $\Fdec$.
    \label{labellings}
\end{remark}

\begin{example} One should think of the simplices of $\Fdec$ as being the same as the simplices of $\mathbb{F}$, except that now each "layer" of edges in the forest is partitioned in some fashion. As an example, we can consider the following example of a 3-simplex in $\Fdec$
$$(A_1, A_2, A_3) \xrightarrow{(\beta_1,F_1)} (B_1,B_2,B_3) \xrightarrow{(\beta_2,F_2)} (C_1, C_2) \xrightarrow{(\beta_3,F_3)} D_1,$$
where each $(\beta_i, F_i)$ can be graphically represented as shown below:

 
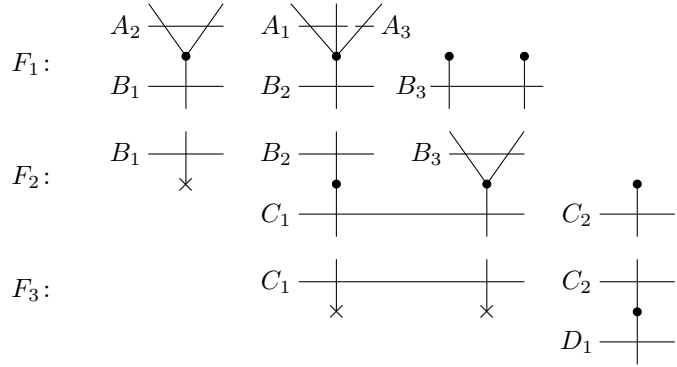
\vspace{1em}

    \begin{center}
    \begin{tikzpicture}
    \draw (-1,0)--(-1,-0.7);
     \filldraw[black] (-1,0) circle (1.5pt);
     \draw (-{0.7*sin(45)-1},0.7) -- (-1,0) -- ({0.7*sin(45)-1},0.7);
     \draw (1,0.7) -- (1,0) -- (1,-0.7);
     \filldraw[black] (1,0) circle (1.5pt);
     \draw (-{0.7*sin(60)+1},0.7) -- (1,0) -- ({0.7*sin(60)+1},0.7);
      \draw (2.5,0) -- (2.5,-0.7);
     \filldraw[black] (2.5,0) circle (1.5pt);
     \draw (3.5,0) -- (3.5,-0.7);
     \filldraw[black] (3.5,0) circle (1.5pt);
     \draw[black] (-1.5,-0.4) -- (-0.5,-0.4);
     \draw[black] (0.5,-0.4) -- (1.5,-0.4);
     \draw[black] (0.5,0.4) -- (1.15,0.4);
     \draw[black] (1.25,0.4) -- (1.5,0.4);
     \draw[black] (2.25,-0.4) -- (3.75,-0.4);
     \draw[black] (-1.5,0.4) -- (-0.5,0.4);
     \node [label={[xshift=-3cm, yshift=-0.5cm]$F_1\colon$}] {};
     \node [label={[xshift=-1.8cm, yshift=0cm]$A_2$}] {};
     \node [label={[xshift=-1.8cm, yshift=-0.8cm]$B_1$}] {};
     \node [label={[xshift=0.2cm, yshift=0cm]$A_1$}] {};
     \node [label={[xshift=1.8cm, yshift=0cm]$A_3$}] {};
     \node [label={[xshift=0.2cm, yshift=-0.8cm]$B_2$}] {};
     \node [label={[xshift=2cm, yshift=-0.8cm]$B_3$}] {};

     \draw (-1,-1.7) -- (-1,-1);
     \draw (-1,-1.7) node {$\times$};
     \draw (1,-1) -- (1,-1.7) -- (1,-2.4);
     \filldraw[black] (1,-1.7) circle (1.5pt);
      \draw (3,-1.7) -- (3,-2.4);
     \filldraw[black] (3,-1.7) circle (1.5pt);
      \draw (-{0.7*sin(45)+3},-1) -- (3,-1.7) -- ({0.7*sin(45)+3},-1);
     \draw (5,-1.7) -- (5,-2.4);
     \filldraw[black] (5,-1.7) circle (1.5pt);
     \draw[black] (0.5,-2.1) -- (3.5,-2.1);
     \draw[black] (0.5,-1.3) -- (1.5,-1.3);
     \draw[black] (4.5,-2.1) -- (5.5,-2.1);
     \draw[black] (2.5,-1.3) -- (3.5,-1.3);
     \draw[black] (-1.5,-1.3) -- (-0.5,-1.3);
     \node [label={[xshift=-3cm, yshift=-2cm]$F_2\colon$}] {};
     \node [label={[xshift=-1.8cm, yshift=-1.7cm]$B_1$}] {};
     \node [label={[xshift=0.2cm, yshift=-1.7cm]$B_2$}] {};
     \node [label={[xshift=0.2cm, yshift=-2.5cm]$C_1$}] {};
     \node [label={[xshift=4.2cm, yshift=-2.5cm]$C_2$}] {};
     \node [label={[xshift=2.2cm, yshift=-1.7cm]$B_3$}] {};

      \draw (1,-2.7) -- (1,-3.4);
     \draw (1,-3.4) node {$\times$};
     \draw (3,-2.7)--(3,-3.4);
     \draw (3,-3.4) node {$\times$};
     \filldraw[black] (5,-3.4) circle (1.5pt);
     \draw (5,-2.7)--(5,-3.4) --(5,-4.1);
     \draw[black] (4.5,-3.8) -- (5.5,-3.8);
     \draw[black] (4.5,-3) -- (5.5,-3);
     \draw[black] (0.5,-3) -- (3.5,-3);
     \node [label={[xshift=-3cm, yshift=-3.5cm]$F_3\colon$}] {};
     \node [label={[xshift=0.2cm, yshift=-3.4cm]$C_1$}] {};
     \node [label={[xshift=4.2cm, yshift=-3.4cm]$C_2$}] {};
     \node [label={[xshift=4.2cm, yshift=-4.2cm]$D_1$}] {};
\end{tikzpicture}

\end{center}

\vspace{1em}

The horizontal lines correspond to the partitions of the leaves and edges of the respective forests, and the omitted information regarding the $\beta_i$ can be easily read off from the picture: for instance, $\beta_2 \colon \langle 3 \rangle \to \langle 2 \rangle$ is given by $\beta_2(1) = \ast$ and $\beta_2(2)=\beta_2(3)=1$. The forest $F$ associated to such a 3-simplex is shown below, together with the corresponding partition of the edges:

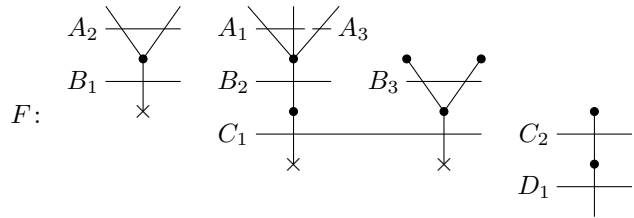
\vspace{1em}

\begin{center}
    \begin{tikzpicture}
    \draw (0,1.4) -- (0,0.7);
    \draw (-{0.7*sin(45)},2.1) -- (0,1.4) -- ({0.7*sin(45)},2.1);
     \draw (0,0.7) node {$\times$};
     \filldraw[black] (0,1.4) circle (1.5pt);
     \draw (2,2.1) -- (2,1.4) -- (2,0.7);
    \draw (-{0.7*sin(60)+2},2.1) -- (2,1.4) -- ({0.7*sin(60)+2},2.1);
     \filldraw[black] (2,0.7) circle (1.5pt);
     \filldraw[black] (2,1.4) circle (1.5pt);
     \draw (2,0.7) -- (2,0);
     \draw (2,0) node {$\times$};
    \draw (-{0.7*sin(45)+4},1.4) -- (4,0.7) -- ({0.7*sin(45)+4},1.4);
    \draw (4,0.7)--(4,0);
    \draw (4,0) node {$\times$};
     \filldraw[black] (4,0.7) circle (1.5pt);
     \filldraw[black] (-{0.7*sin(45)+4},1.4) circle (1.5pt);
     \filldraw[black] ({0.7*sin(45)+4},1.4) circle (1.5pt);
     \draw (6,0.7) -- (6,0)--(6,-0.7);
     \filldraw[black] (6,0) circle (1.5pt);
     \filldraw[black] (6,0.7) circle (1.5pt);
     \draw[black] (-0.5,1.1) -- (0.5,1.1);
     \draw[black] (1.5,1.1) -- (2.5,1.1);
     \draw[black] (-0.5,1.8) -- (0.5,1.8);
     \draw[black] (1.5,1.8) -- (2.15,1.8);
     \draw[black] (2.25,1.8) -- (2.5,1.8);
     \draw[black] (3.5,1.1) -- (4.5,1.1);
     \draw[black] (1.5,0.4) -- (4.5,0.4);
     \draw[black] (5.5,-0.3) -- (6.5,-0.3);
     \draw[black] (5.5,0.4) -- (6.5,0.4);
     \node [label={[xshift=-1.5cm, yshift=0.3cm]$F\colon$}] {};
     \node [label={[xshift=-0.8cm, yshift=0.7cm]$B_1$}] {};
     \node [label={[xshift=-0.8cm, yshift=1.4cm]$A_2$}] {};
     \node [label={[xshift=1.2cm, yshift=0.7cm]$B_2$}] {};
     \node [label={[xshift=1.2cm, yshift=0cm]$C_1$}] {};
     \node [label={[xshift=1.2cm, yshift=1.4cm]$A_1$}] {};
     \node [label={[xshift=2.8cm, yshift=1.4cm]$A_3$}] {};
     \node [label={[xshift=3.2cm, yshift=0.7cm]$B_3$}] {};
     \node [label={[xshift=5.2cm, yshift=-0.7cm]$D_1$}] {};
     \node [label={[xshift=5.2cm, yshift=0cm]$C_2$}] {};
\end{tikzpicture}
\end{center}

\label{partitioned_forest}
\end{example}

As Example \ref{partitioned_forest} shows, we can think of a general $k$-simplex of $\Fdec$ as a forest $F \in N\mathbb{F}_k$ together with a partition of its edges, and the face/degeneracy maps of $N \Fdec$ operate on the forest $F$ in the same as they do in $N\mathbb{F}$. As we will be particularly concerned with how these maps act, we will often omit the partition information of a $k$-simplex in $N\mathbb{F}$ and only write the associated forest $F$.

\begin{definition}
    Let $X$ be a dendroidal set. The \textit{symmetric monoidal envelope of $X$} is the simplicial set $\Proptensor(X)^\otimes$ with set of  $k$-simplices given by
    $$ \Proptensor(X)^\otimes_k = \coprod_{ F \in N \Fdec_k} X_{F},$$
    where we write $X_F = X_{T_1} \times \cdots \times X_{T_n}$ whenever $F$ is the forest $\bigoplus_{i=1}^n T_i$.

    The simplicial face and degeneracy maps act on $\Proptensor(X)^\otimes$ via the simplicial structure on the indexing simplicial set $\Fdec$. Here we are also using that the simplicial morphisms act on the forests $F \in N\mathbf{F}_k$ via dendroidal maps.
    \label{Infinity_Prop}
\end{definition}

Letting $\mathsf{part}\colon N\Fdec \to N\Fin_\ast$ be the functor that retains the information about the partitions, we can see $\Prop(X)^\otimes$ as a simplicial set over the nerve of finite pointed sets via the composite of 
$$ \Proptensor(X)^\otimes \longrightarrow N\Fdec \xlongrightarrow{\mathsf{part}} N\Fin_\ast,$$ 
where $\Proptensor(X)^\otimes \to N\Fdec$ just projects onto the respective component of the coproduct defining $N\Fdec$.

\begin{proposition}
    Let $X$ be a dendroidal $\infty$-operad. Then the functor $\Proptensor(X)^\otimes \to N\Fin_\ast$ exhibits $\Proptensor(X)^\otimes$ as a symmetric monoidal $\infty$-category in the sense of Lurie.
    \label{prop_x_cat}
\end{proposition}

\begin{proof}
    Letting $\Phi$ denote the functor $\Proptensor(X)^\otimes \to N\Fin_\ast$, we need to check the following properties:
    \begin{enumerate}[label=(\alph*)]
        \item $\Phi$ is an inner fibration, which is equivalent to showing that $\Proptensor(X)^\otimes$ is an $\infty$-category.
        \item $\Phi$ is a cocartesian fibration.
        \item For each $n \geq 1$, the product of the functions $\rho^j \colon \langle n \rangle \to \langle 1 \rangle$, defined for each $1 \leq j \leq n$ and determined by the condition $(\rho^j)^{-1}(1)=\{j\}$, induces a map
        $$ \Proptensor(X)^\otimes_{\langle n \rangle} \longrightarrow \prod_{j=1}^n \Proptensor(X)^\otimes_{\langle 1 \rangle}$$
        which is an equivalence of $\infty$-categories.
    \end{enumerate}

   For (a) we need to construct horn fillers
   $$
        \begin{tikzcd}
{\Lambda^j[k]} \arrow[r] \arrow[d] & \Proptensor(X)^\otimes \\
{\Delta[k]} \arrow[ru, dashed]     &        
\end{tikzcd}
$$
    for the inner horn inclusions, that is, when $0<j<k$. We begin by observing that, by the construction of $\Proptensor(X)^\otimes$, there is a similar associated extension problem for the nerve of $\Fdec$
    $$ 
    \begin{tikzcd}
{\Lambda^j[k]} \arrow[r] \arrow[d] & N\Fdec, \\
{\Delta[k]} \arrow[ru, dashed]     &        
\end{tikzcd}.
    $$
    which has a (unique) solution, since it relates to the nerve of a category. As we have already discussed, one can think of this newly constructed $k$-simplex in $N\Fdec$ as a forest $F$, with each layer of edges partitioned in some fashion. Using this language, we are left with arguing that the following diagram in $\dSets$ 
    $$
\begin{tikzcd}
\Lambda^j F \arrow[r] \arrow[d] & X \\
F \arrow[ru, dashed]                                                                               &  
\end{tikzcd},
    $$
admits the dashed lift, where $\Lambda^j F$ are the horns of Definition \ref{definition_horns}. By Lemma \ref{horn_inner} we have that  $\Lambda^j F \to F$ is a dendroidal inner anodyne, and consequently the desired morphism exists due to $X$ being an $\infty$-operad.

For (b), suppose we are given a map of finite pointed sets $\beta\colon \langle n \rangle \to \langle k \rangle$ together with a lift of its source to $\Proptensor(X)^\otimes$. Such a lift corresponds to an object $(A_1, \ldots, A_n)$ in $\Fdec$, together with an $A$-tuple $ \vec{c} = (c_1, \ldots, c_A) \in X_\eta^{ \times A}$, where $A = \lvert A_1 \rvert + \cdots + \vert A_n \rvert$, coming from evaluating $X$ on the forest $\bigoplus_{i=1}^n A_i$. We also write $A_\ast$ for the sum of the cardinalities $\lvert A_i \rvert$ ranging over $i \in \alpha^{-1}(\ast)$.

We first define a lift $(B_1, \ldots, B_k)$ to $\Fdec$ of the finite pointed set $\langle k \rangle$: simply set $B_j = \bigoplus_{ i \in \beta^{-1}(j) } A_i$ for each $1 \leq j \leq k$, which is the empty forest whenever $\beta^{-1}(j) = \emptyset$. Let also $F \in N\mathbb{F}_1$ be the forest obtained by first extending the edges of $\bigoplus_{j=1}^k B_j$ to 1-corollas, and then adjoining $A_\ast$ components of uprooted trees of the form $\eta$. 

Having constructed the pair $(\beta, F)$, we define the desired 1-simplex $\tilde{\beta}$ in  $\Proptensor(X)^\otimes$ lifting $\beta$ as follows:
     \begin{enumerate}[label=(\roman*)]
         \item $\tilde{\beta}$ is indexed by the morphism $ (\beta,F) \colon (A_1, \ldots, A_n) \longrightarrow (B_1, \ldots, B_k)$ in $\Fdec$ that we defined in the pragraph above.
         \item For each $1 \leq j \leq k$, the value of $\tilde{\beta}$ on the 1-corollas coming from $B_j$ is the image via the dendroidal map induced from the degeneracy $C_1 \to C_0$
         $$ \prod_{i \in \beta^{-1}(j)} X_{A_i} \longrightarrow \prod_{i \in \beta^{-1}(j)} X_{C_1}^{ \times \lvert A_i \rvert}$$
         of the subtuple of $\vec{c}$ indexed by the elements of $A_i$.
         \item The value of $\tilde{\beta}$ on $A_{\ast} \cdot C_1$ is given by the remaining colours of $\vec{c}$ not mentioned in the previous case.
     \end{enumerate}

     We claim that $\tilde{\beta}$ is a $\Phi$-cocartesian lift of $\beta$, which comes down to checking that there is a diagonal lift in the diagram
     $$
\begin{tikzcd}
{\Lambda^0[m]} \arrow[d] \arrow[r] & \Proptensor(X)^\otimes \arrow[d, "\Phi"] \\
{\Delta[m]} \arrow[r]              & N\Fin_\ast                     
\end{tikzcd}
     $$
for $m \geq 2$, whenever the top horizontal map sends the initial edge $\Delta^{\{0,1\}}$ to $\tilde{\beta}$. Unravelling the definitions and in a similar fashion to what happened when proving part (a), this translates into showing the existence in $\dSets$ of the dashed arrow $\tilde{x}$ in 
$$
\begin{tikzcd}
\Lambda^0 G \arrow[r, "x"] \arrow[d] & X, \\
G \arrow[ru, dashed, "\tilde{x}"]            &  
\end{tikzcd}
$$
where $G \in N\mathbb{F}_m$ is a forest. Moreover, the initial edge of this $m$-simplex $(\partial_2 \cdots \partial_{m-1} \partial_m) G \in N\mathbb{F}_1$ is $F$, and the induced map $X_G \to X_F$ should send  $\tilde{x}$ to $\tilde{\beta}$.

We can assume that the forest $F$ only contains copies of $C_1$, or equivalently that $\beta^{-1}(\ast) = \ast$: these components coming from the fiber of the basepoint will only contribute to $G$ with uprooted trees of the form $\eta$, and for such components the extension problem is trivial. Assuming this, one can check that the desired extension is the image of $x$ via the composition of
    $$ X_{\Lambda^0 G} \xlongrightarrow{(\partial_{m-1} G \to \Lambda^0 G)^*} X_{\partial_{m-1} G} \xlongrightarrow{(\sigma_m \partial_{m-1} G \to \partial_{m-1} G)^*} X_{\sigma_m \partial_1 G} = X_G,$$
where $\sigma_m \colon N\Fdec_{m-1} \to N\Fdec_{m}$ denotes the last degeneracy map. In the last step we used the identification $\sigma_m \partial_1 G = G$, which is due to $F$ being a forest of 1-corollas.

     For part (c), we have an isomorphism of categories
     \begin{equation}
       \Fdec_{\langle n \rangle} \xlongrightarrow{\cong} \prod_{i=1}^n \Fdec_{\langle 1 \rangle}.
       \label{proj}
     \end{equation}
     Here $\Fdec_{\langle n \rangle}$ is the subcategory of $\Fdec$ over $\langle n \rangle$: a $k$-simplex in the nerve of this category is a forest of linear trees which has been partitioned into $n$ subforests, and \eqref{proj} keeps track of each of these subforests. Using this description, it readily follows that the induced map $$\Proptensor(X)^\otimes_{\langle n \rangle} \longrightarrow \prod_{i=1}^n \Proptensor(X)^\otimes_{\langle 1 \rangle}$$
     is an equivalence, as we wanted to show.
\end{proof}

\begin{proposition}
    Let $\mathbf{P}$ be a coloured operad. Then $\Proptensor(N\mathbf{P})^\otimes$ is the symmetric monoidal $\infty$-category associated to the (strict) symmetric monoidal category $\Prop(\mathbf{P})$.
    \label{Prop_P}
\end{proposition}

\begin{proof}
    One can easily describe the associated symmetric monoidal category $\Proptensor(N \mathbf{P})^\otimes_{\langle 1 \rangle}$ from Definition \ref{Infinity_Prop} (note that since we are working over the object $\langle 1 \rangle$, the partition information does not exist):
    \begin{itemize}[label=$\diamond$]
    \item The set of objects of $\Proptensor(N \mathbf{P})^\otimes_{\langle 1 \rangle}$ is the coproduct $\coprod_{n \geq 0} N\mathbf{P}_\eta^{\times n},$
    and consequently an object is just a tuple of colours $(c_1, \ldots, c_n)$ of $\mathbf{P}$, for $n \geq 0$.
    \item A morphism $(c_1, \ldots, c_m) \to (d_1, \ldots, d_n)$ will be given by an $n$-tuple
    $$ (p_1, \ldots, p_n) \in X_{C_{k_1}} \times \cdots \times X_{C_{k_n}},$$
    for integers $k_i \geq 0$ satisfying $\sum_i k_i = m$. Moreover, the $c_i$'s and $d_j$'s are the values of $X$ on the leaves and roots of these corollas, respectively. 
    
This can be equivalently encoded by a function of sets $f \colon \underline{m} \to \underline{n}$ where $f(i)=j$ if $c_i$ and $d_j$ belong to the same corolla, and operations $p_j$ in $\mathbf{P}$ with output $d_j$ and inputs given by the $c_i$'s satisfying $f(i)=j$ ordered according to $\underline{n}$.

\item By the proof of part (c) of Proposition \ref{prop_x_cat}, the monoidal structure is induced by the concatenation of forests $(F, G) \mapsto F \oplus G$. On $\Proptensor(N \mathbf{P})^\otimes_{\langle 1 \rangle}$ this corresponds to the concatenation of tuples of colours.
\end{itemize}

A quick comparison with Definition \ref{Prop_Strict} readily shows that this coincides with $\Prop(\mathbf{P})$. It is also clear that $\Proptensor(N \mathbf{P})^\otimes$ is the usual construction of a symmetric monoidal $\infty$-category from a strict symmetric monoidal category, finishing the proof.
\end{proof}

\bibliographystyle{alpha}
\bibliography{bibliography.bib}
\end{document}